\documentclass[12pt]{amsart}
\calclayout

\title[Bergman local isometries are biholomorphisms]{Bergman local isometries are biholomorphisms}
\author[J. Yum]{Jihun Yum}
\address{Research Institute of Molecular Alchemy, Gyeongsang National University, Jinju, 52828, Republic of Korea}
\email{jihun0224@gnu.ac.kr}
\date{\today}
\usepackage{amsthm,amssymb,latexsym}

\usepackage[abbrev]{amsrefs}
\usepackage{hyperref}

\usepackage[autostyle]{csquotes}
\usepackage{xcolor}
\usepackage{cite}

\usepackage{chemarrow}

\usepackage{kotex}
\usepackage{tikz-cd}
\newcommand{\RR}{\mathbb{R}}		% the real number R
\newcommand{\CC}{\mathbb{C}} 		% Complex C
\newcommand{\BB}{\mathbb{B}}        % unit ball B
\newcommand{\NN}{\mathbb{N}} 		% Natural number N
\renewcommand{\O}{\Omega}			% domain
\newcommand{\Prob}{\mathcal{P}}     % Probability space
\newcommand{\Bergman}{\mathcal{B}}  % Bergman kernel
\newcommand{\Dia}{\mathcal{D}}      % Diastasis function
\newcommand{\Poisson}{P}            % Poisson-Bergman kernel

\theoremstyle{plain}
\newtheorem{thm}{Theorem}[section]

\newtheorem{prop}[thm]{Proposition}

\newtheorem{mainthm}{Theorem}

\theoremstyle{definition}

\theoremstyle{remark}
\newtheorem{rmk}[thm]{Remark}

\theoremstyle{property}

\numberwithin{equation}{section}

\begin{document}
	
	\maketitle

\begin{abstract}
    We prove that a proper holomorphic local isometry between bounded domains with respect to the Bergman metrics is necessarily a biholomorphism. 
    The proof relies on a new method grounded in Information Geometry theories.
\end{abstract}

\section{Introduction}

Let $\Omega_1$ and $\Omega_2$ denote bounded domains in $\mathbb{C}^n$ equipped with the Bergman metrics $g_{B1}$ and $g_{B2}$, respectively. If a biholomorphism $f \colon \Omega_1 \rightarrow \Omega_2$ exists, then according to the transformation formula for the Bergman kernels (\ref{equ: transformation formula}), it is well-established that $f$ induces an isometry with respect to the Bergman metric, i.e., $f^* g_{B2} = g_{B1}$. This article explores the converse implication. 
The main theorem is the following.

\begin{mainthm} \label{thm: theorem A}
Let $\O_1$ and $\O_2$ be bounded domains in $\CC^n$. For a proper holomorphic map $f\colon \Omega_1 \rightarrow \Omega_2$,
if $f^* g_{B2} = \lambda g_{B1}$ holds on an open subset $U \subset \Omega_1$ for some constant $\lambda > 0$, then $f$ is a biholomorphism.
\end{mainthm}

Theorem~\ref{thm: theorem A} extends renowned Lu's uniformization theorem(\cite{Lu1966onkahler}), which asserts that if a bounded domain $\Omega$ in $\mathbb{C}^n$ admits the complete Bergman metric with constant holomorphic sectional curvature, then $\Omega$ is biholomorphic to the unit ball. To elucidate this, we briefly outline a proof of Lu's theorem. For a detailed exposition, we refer readers to (Theorem 4.2.2, \cite{greene2011geometry}). Assuming the holomorphic sectional curvature is negatively constant by Myers' theorem, the universal cover $\widetilde{\Omega}$ of $\Omega$ is biholomorphic to the unit ball, and the \emph{covering map $f\colon \mathbb{B}^n \rightarrow \Omega$ constitutes a Bergman local isometry}. 
This implies that $f$ is a $\mathbb{C}$-linear map in the Bergman representative coordinates at some point $p \in \mathbb{B}^n$. In other words, the following diagram 
\begin{equation*} 
    \begin{tikzcd}
        \BB^n \arrow[swap]{d}{f} \arrow{r}{\text{rep}_p} &  \arrow{d}{A} \CC^n  \\
        \O \arrow[swap]{r}{\text{rep}_{f(p)}} & \CC^n
    \end{tikzcd}
\end{equation*}
commutes locally on a neighborhood of $p \in \BB^n$, where $A$ represents a $\CC$-linear map.
Consequently, a left inverse $g\colon \Omega \rightarrow \mathbb{B}^n$ of $f$ exists, implying injectivity of $f$. The significance of the unit ball $\mathbb{B}^n$ lies in its properties: (1) $\text{rep}_p : \mathbb{B}^n \rightarrow \mathbb{C}^n$ is $\mathbb{C}$-linear (and thus invertible) and (2) $\mathbb{B}^n$ is bounded (to apply the Riemann removable singularity theorem). Building upon this idea, S. Yoo (\cite{yoo2017differential}) extended Lu's theorem by demonstrating that for a bounded domain $\Omega \subset \mathbb{C}^n$ admitting a Bochner ``pole'', a holomorphic Bergman local isometry $f$ from $\Omega$ onto a complex manifold $M$ is necessarily a biholomorphism. 
%Numerous results related to this problem were challenging to generalize beyond the unit ball. 
We emphasize that Theorem~\ref{thm: theorem A} not only escapes the unit ball but also imposes no additional conditions (e.g., pseudoconvexity or Bergman completeness) apart from boundedness.

In a similar vein, M. Skwarczynski (IV.17. Theorem \cite{skwarczynski1980biholomorphic}) established that if $f$ respects the transformation rule of the Bergman kernel (\ref{equ: transformation formula}), then $f$ becomes a biholomorphism under the additional assumption that $\Omega_1$ and $\Omega_2$ are complete with respect to the Skwarczynski distance $\rho$ defined in his thesis (\cite{skwarczynski1980biholomorphic}). Furthermore, he demonstrated that $\rho(z,w) = \rho(f(z),f(w))$ implies the injectivity of $f$. Notably, these assumptions concerns the Bergman kernel, which implies the assumption of Theorem \ref{thm: theorem A} through differentiation. Hence his results follow from Theorem~\ref{thm: theorem A} immediately without any completeness assumption.

The proof of Theorem \ref{thm: theorem A} relies on a novel method grounded in Information Geometry theories, particularly the Factorization Theorem (Theorem \ref{thm: equivalent conditions for sufficiency}) and the result (Theorem \ref{thm: sufficient implies 1-1}) established by G. Cho and the author.

\subsection*{The Acknowledgments}
The author would like to express his gratitude to Prof. Kang-Tae Kim for his insightful suggestion to approach Lu's theorem from a statistical perspective.
The author extends his thanks to Prof. Kyeong-Dong Park, Sungmin Yoo, and Ye-Won Luke Cho for their valuable comments and words of encouragement.
Additionally, the author acknowledges Hoseob Seo for engaging in profound discussions.

The author is supported by Learning $\&$ Academic research institution for Master’s·PhD students, and Postdocs(LAMP) Program of the National Research Foundation of Korea(NRF) grant funded by the Ministry of Education(No. RS-2023-00301974).

%---------------------------------------------------------------------------------

\section{Preliminaries}

\subsection{Information Geometry}

    We briefly introduce the basic definitions of Information Geometry. 
    For more details, we refer the readers to \cite{amari2000methods} and \cite{ay2017information_book}.

    Let $\Xi \subset \RR^m$ be a domain and $dV$ be the standard Lebesgue measure on $\RR^m$ (more generally a set $\Xi$ can be an arbitrary measurable space, but in this paper, we focus on a domain $\Xi \subset \RR^m$ with the Borel $\sigma$-algebra).
    Let $\Prob(\Xi)$ be the space of all probability measures dominated by $dV$. 
    In general, the space $\Prob(\Xi)$ is infinite-dimensional and a subset of the Banach space of all signed measures on $\Xi$ with the total variation. From this Banach space, one can induce a smooth structure on $\Prob(\Xi)$.
    Then a triple $(\O, \Xi, \Phi)$ (or $\Phi: \O \hookrightarrow \Prob(\Xi))$, where $\O \subset \RR^n$ is domain and $\Phi: \O \rightarrow \Prob(\Xi)$ is a smooth embedding, is called a {\it statistical manifold} (or {\it statistical model}).
    We call $\O$ a {\it parameter space} and $\Xi$ a {\it sample space}, respectively. Note that the dimensions $n$ and $m$ of the two spaces are independent.

    On $\Prob(\Xi)$, there exists a natural pseudo-Riemannian metric $g_F$ called the Fisher information metric.
    By using $\Phi$ as a (global) chart with a coordinate system $(x_1, \dots, x_n)$, the Fisher information metric $g_F(x) = \sum_{\alpha,\beta=1}^n g_{\alpha\beta}(x) dx_{\alpha} \otimes dx_{\beta}$ restricted on $\Phi(\O)$ can be written as
    \begin{equation*}
		g_{\alpha\beta}(x) := \int_{\Xi} (\partial_{\alpha} \log P(x,\xi))(\partial_{\beta} \log P(x,\xi)) P(x, \xi) dV(\xi),
	\end{equation*}
    where $\Phi(x) = P(x, \cdot)dV(\cdot)$ is a probability measure on $\Xi$, and $\partial_{\alpha} := \frac{\partial}{\partial x_{\alpha}}$, $\partial_{\beta} := \frac{\partial}{\partial x_{\beta}}$ for $\alpha, \beta = 1, \dots, n$.

    One easy (but remarkable) example of a statistical manifold and the Fisher information metric is the set $\mathcal{N}$ of all Gaussian normal distributions on $\RR$. Since each element in $\mathcal{N}$ is uniquely characterized by the mean $\mu \in \RR$ and the standard deviation $\sigma > 0$, $\mathcal{N}$ can be parametrized by the upper-half space $\mathbb{H}$ in $\RR^2$, i.e., $\mathbb{H} \hookrightarrow \mathcal{N} \subset \Prob(\RR)$. Then $\mathcal{N}$ with the Fisher information metric $g_F$ becomes the hyperbolic space with negatively constant (Gaussian) curvature.

\subsection{Sufficient Statistics}

    Let $\Phi_1: \O_1 \hookrightarrow \Prob(\Xi_1)$ be a statistical manifold.
    For domains $\Xi_1 \subset \RR^{m_1}$ and $\Xi_2 \subset \RR^{m_2}$, a (Borel)-measurable function $f: \Xi_1 \rightarrow \Xi_2$ is called a {\it statistic}.
    Given a surjective statistic $f: \Xi_1 \rightarrow \Xi_2$, one can induce the natural map $\kappa: \Prob(\Xi_1) \rightarrow \Prob(\Xi_2)$ defined by the \emph{measure push-forward} of $f$, i.e., for $\mu \in \Prob(\Xi_1)$, 
    \begin{align} \label{def: measure push-forward}
        \kappa(\mu)(B) := \mu(f^{-1}(B))
    \end{align}
    for each Borel subset $B \subset \Xi_2$. 
    Then this induces a map $\kappa \circ \Phi_1: \O_1 \rightarrow \Prob(\Xi_2)$ as follows.
    \begin{equation} \label{dia: measure push-forward}
        \begin{tikzcd}
            \O_1 \arrow{r}{\Phi_1} \arrow[swap]{dr}{\kappa \circ \Phi_1 }   &  \arrow{d}{\kappa} (\Prob(\Xi_1), g_{F_1})  \\
            & (\Prob(\Xi_2), g_{F_2})
        \end{tikzcd}
    \end{equation}
    
    Now, we compare two Fisher information metrics $g_{F_1}$ and $g_{F_2}$ of $\Prob(\Xi_1)$ and $\Prob(\Xi_2)$, respectively.  
    Interestingly, the Fisher information metric always satisfies the monotone decreasing property, that is,
    \begin{align*}
        (\kappa \circ \Phi_1)^* g_{F_2} (X,X) \le \Phi_1^* g_{F_1} (X,X)
    \end{align*}
    for all $p \in \O_1$ and $X \in T_p(\O_1)$.
    When the inequality becomes equality, a statistic $f: \Xi_1 \rightarrow \Xi_2$ is called {\it sufficient} for $\Phi_1(\O_1)$. 
    In other words, sufficient statistics preserve the geometric structure of statistical manifolds.

    The following characterization theorem for sufficiency is the first ingredient for the proof of Theorem~\ref{thm: theorem A}.  
    
    \begin{thm}[cf. Theorem 2.1, \cite{amari2000methods}](Factorization Theorem)  \label{thm: equivalent conditions for sufficiency}
    The following are equivalent.
    \begin{enumerate}
            \item A statistic $f: \Xi_1 \rightarrow \Xi_2$ is sufficient for $\Phi_1(\O_1)$.
            \item   The measurable function
            \begin{equation*}
                r(x, \xi) := \frac{P(x,\xi)}{Q(x, f(\xi))}  \quad dV(\xi)-\text{a.e.}
            \end{equation*}
            does not depend on $x \in \O$, where $P(x, \xi)dV(\xi) \in \Phi_1(\O_1)$ and $\kappa(P(x, \xi)dV(\xi))(x, \zeta) := Q(x, \zeta) \kappa(dV)(\zeta)$.
            \item 
            For each $x \in \O_1$, there exist functions $s(x, \cdot) \in L^1(\Xi_2, \kappa(dV))$ and $t \in L^1(\Xi_1, dV)$ such that 
            \begin{equation*}
                P(x, \xi) = s(x, f(\xi)) t(\xi)  \quad dV(\xi)-\text{a.e.},
            \end{equation*} 
            where $P(x, \xi)dV(\xi) \in \Phi_1(\O_1)$.
    \end{enumerate}
    \end{thm}

    For the proof, we refer the readers to \cite{ay2017information_book}. 
    In \cite{ay2017information_book}, Proposition 5.5 and Proposition 5.6 show that $(3) \Leftrightarrow (2)$ and $(2) \Leftrightarrow (1)$, respectively.

    \begin{rmk}
        In condition (2) of Theorem~\ref{thm: equivalent conditions for sufficiency}, for each $x \in \O_1$, automatically $Q(x, \cdot) \in L^1(\Xi_2, \kappa(dV))$ by the definition of the measure push-forward and Radon--Nikodym theorem. 
        Also, one can show that $r(x, \cdot) \in L^1(\Xi_1, dV)$ as follows.
        \begin{align*}
            \int_{\O_1} r(x, \xi) dV(\xi) 
            &= \int_{\O_1} \frac{P(x,\xi)}{Q(x, f(\xi))} dV(\xi) 
            = \int_{f(\O_1)} \frac{1}{Q(x, \zeta)} \kappa(P(x, \xi)dV(\xi)) \\
            &= \int_{f(\O_1)} \kappa(dV)
            = \int_{\O_1} dV(\xi) < \infty. 
        \end{align*}
    \end{rmk}

    \begin{rmk} \label{rmk: 1-1 implies sufficient}
        It is easy to see that an injective statistic $f$ is always sufficient from condition (3) of Theorem~\ref{thm: equivalent conditions for sufficiency}: choose $s(x, \zeta) = P(x, f^{-1}(\zeta))$ and $t(\xi) = 1$. \\
    \end{rmk}

\subsection{Bergman Geometry}

    Let $\O$ be a bounded domain in $\CC^n$ and $A^2(\O)$ be the set of all $L^2$ holomorphic functions on $\O$. Then $A^2(\O)$ is a separable Hilbert space with the inner product given by
    \begin{align*}
        \left< f, g \right> := \int_{\O} f(z) \overline{g(z)} dV(z).
    \end{align*}
    The {\it Bergman kernel} function $\Bergman : \O \times \O \rightarrow \CC$ is defined by
    $$ \Bergman(z,\xi) := \sum_{j=0} s_j(z) \overline{s_j(\xi)},  $$
    where $\{s_j\}_{j=0}$ is a complete orthonormal basis for $A^2(\O)$. 
    Note that the definition is independent of the choice of an orthonormal basis.
    The {\it Bergman metric} $g_B(z) = \sum_{\alpha,\beta=1}^n g_{\alpha \overline{\beta}}(z) dz_{\alpha} \otimes d\overline{z}_{\beta}$ on $\O$ is defined by 
    \begin{align*}
        g_{\alpha \overline{\beta}}(z)
        := \frac{\partial^2}{\partial z_{\alpha} \partial \overline{z}_{\beta}} \log \Bergman(z,z), 
    \end{align*}
    provided that $\Bergman(z,z) > 0$ on $\O$.
    For a bounded domain $\O$, it is well-known that $g_B$ is well-defined and positive-definite.
    
    If $f: \O_1 \rightarrow \O_2$ is a biholomorphism between bounded domains $\O_1, \O_2 \subset \CC^n$, the following transformation formula for the Bergman kernels holds:
    \begin{align} \label{equ: transformation formula}
        \Bergman_1(z,\xi) = J_{\CC}f(z) \cdot \Bergman_2(f(z),f(\xi)) \cdot \overline{J_{\CC}f(\xi)},
    \end{align}
    where $\Bergman_1$ and $\Bergman_2$ are the Bergman kernels of $\O_1$ and $\O_2$, respectively, and $J_{\CC}f$ is the determinant of the complex Jacobian matrix of $f$. Moreover, from (\ref{equ: transformation formula}), $f$ becomes an isometry with respect to the Bergman metric, which is one of the most important property of the Bergman metric.

\subsection{Bounded domains as statistical manifolds}

    For a bounded domain $\O_1 \subset \CC^n$, G. Cho and the author (\cite{cho2023statistical}) constructed the map $\Phi_1: \O_1 \rightarrow \Prob(\O_1)$ defined by
    \begin{align*}
        \Phi_1(z) := P_1(z,\xi)dV(\xi) := \frac{|\Bergman_1(z, \xi)|^2}{\Bergman_1(z,z)} dV(\xi).
    \end{align*}
    Then they proved that $\Phi_1: \O_1 \rightarrow \Prob(\O_1)$ is indeed a statistical manifold and the pull-back of the Fisher information metric on $\Prob(\O_1)$ is the same as the Bergman metric on $\O_1$. 
    We call this $\Phi_1: \O_1 \hookrightarrow \Prob(\O_1)$ a {\it Bergman statistical manifold}. 
    They also present interesting other results in this framework. 

    For two bounded domains $\O_1, \O_2 \subset \CC^n$, let $f: \O_1 \rightarrow \O_2$ be a proper holomorphic map. 
    Note that a proper holomorphic map $f$ is surjective (Proposition 15.1.5, \cite{rudin2012function_book}). 
    It is known by R. Remmert that $V := \{ f(z) : J_{\CC} f(z)=0 \}$ is a complex variety in $\O_2$ and
    \begin{align*} \label{f locally invertible}
        f: f^{-1}(\O_2 \setminus V) \rightarrow (\O_2 \setminus V) \text{ is an } m\text{-sheeted } \text{holomorphic covering map} 
    \end{align*} 
    for some $m \in \NN$.
    In this case, the measure push-forward $\kappa$ of $f$ defined by (\ref{def: measure push-forward}) can be explicitly written as 
    \begin{equation} \label{euq: kappa expression}
        \kappa(P_1(z, \xi) dV(\xi))(z, \zeta) = \sum_{k=1}^{m} \frac{|\Bergman_{1}(z, f^{-1}_k(\zeta) )|^2 |J_{\CC}f^{-1}_k(\zeta)|^2}{\Bergman_{1}(z,z)} dV(\zeta) 
    \end{equation}
    for all $z \in \O_1$ and $\zeta \in \O_2 \setminus V$,
    where $f^{-1}_k$ is a local inverse of $f$ and $J_{\CC}f^{-1}_k(\zeta)$ is the determinant of the complex Jacobian matrix of $f^{-1}_k$.
    For a Bergman statistical manifold, the diagram (\ref{dia: measure push-forward}) becomes
    \begin{equation*} 
        \begin{tikzcd}
            (\O_1, g_{B_1}) \arrow{r}{\Phi_1} \arrow[swap]{dr}{\kappa \circ \Phi_1 }   &  \arrow{d}{\kappa} (\Prob(\O_1), g_{F_1})  \\
            & (\Prob(\O_2), g_{F_2})
        \end{tikzcd}
    \end{equation*}

    \begin{rmk}
        The way of the factorization (3) of Theorem \ref{thm: equivalent conditions for sufficiency} is not unique. For example, in a Bergman statistical manifold, if $f: \O_1 \rightarrow \O_2$ is a biholomorphism, 
        \begin{align*}
        P_1(z, \xi) = \frac{|\Bergman_1(z, \xi)|^2}{\Bergman_1(z,z)} 
        = \frac{|\Bergman_1(z, f^{-1}(f(\xi)))|^2}{\Bergman_1(z,z)} \cdot 1
        \end{align*}
        or
        \begin{align*}
        P_1(z, \xi) = \frac{|\Bergman_1(z, \xi)|^2}{\Bergman_1(z,z)} 
        = \frac{|\Bergman_2(f(z), f(\xi))|^2 |J_{\CC}f(z)|^2}{\Bergman_1(z,z)} \cdot |J_{\CC}f(\xi)|^2.
        \end{align*}
        In the second factorization, the transformation formula for the Bergman kernel (\ref{equ: transformation formula}) is applied.
    \end{rmk}

    Although an injective statistic $f$ is sufficient (Remark~\ref{rmk: 1-1 implies sufficient}), the converse is not true in general.
    The following theorem proved by G. Cho and the author says that the converse is also true for a Bergman statistical manifold, which is our second main ingredient for the proof of Theorem~\ref{thm: theorem A}.

    \begin{thm} [Corollary 4.9, \cite{cho2023statistical}] \label{thm: sufficient implies 1-1}
        Assume that $f: \O_1 \rightarrow \O_2$ is a proper holomorphic map. 
        Then $f$ is injective if and only if $f$ is sufficient for $\Phi_1(\O_1)$.
    \end{thm}

    For the readers' convenience, we briefly summarize the proof of Theorem \ref{thm: sufficient implies 1-1}.

    They first proved that $f$ is sufficient, i.e., $\Phi_1^* g_{F_1} = (\kappa \circ \Phi_1)^* g_{F_2}$ on $\O_1$, if and only if 
    \begin{equation} \label{derivative log B equal}
       \partial_{\alpha} \log \Bergman_{1}(z, f_1^{-1}(\zeta)) 
	= \cdots  
	= \partial_{\alpha} \log \Bergman_{1}(z, f_m^{-1}(\zeta)) 
    \end{equation}
    for all $z \in \O_1$, $\zeta \in \O_2 \setminus V$ and $\alpha = 1, \cdots, n$ by calculating the Fisher information metrics using (\ref{euq: kappa expression}).

    Now suppose, for the sake of contradiction, that there exist $p \neq q \in f^{-1}(\O_2 \setminus V)$ such that $f(p)=f(q)$. Then (\ref{derivative log B equal}) implies 
    \begin{equation} \label{2.4}
        \partial_{\alpha} \log \Bergman_{1}(p,p)  
	= \partial_{\alpha} \log \Bergman_{1}(p,q). 
    \end{equation}
    
    Then, using a special orthonormal basis $\{s_j\}^{\infty}_{j=0}$ for $A^2(\O_1)$ with respect to $p$, they showed that, from (\ref{2.4}), 
    if $s(p)=0$ then $s(q)=0$ for all $s \in A^2(\O_1)$. 
    This contradicts the fact that $A^2(\O_1)$ separates points, that is,
    \begin{equation*}
        \text{for all} \ p \neq q \in \O_1, \ \exists \ s \in A^2(\O_1) \ \text{such that} \ s(p)=0 \ \text{and} \  s(q) \neq 0.
    \end{equation*}
    Therefore they concluded that $f: f^{-1}(\O_2 \setminus V) \rightarrow (\O_2 \setminus V)$ is a injective map, and the properness of $f$ (cf. Theorem 15.1.9, \cite{rudin2012function_book}) implies that in fact $f: \O_1 \rightarrow \O_2$ is injective. 

    \hfill
    $\square$

%~~~~~~~~~~~~~~~~~~~~~~~~~~~~~~~~~~~~~~~~~~~~~~~~~~~~~~~~~~~~~~~~~~~~~~~~~~~~~~

    \section{Proof of Theorem \ref{thm: theorem A}}
        Since the surjectivity of $f$ follows from that $f$ is proper and holomorphic (Proposition 15.1.5, \cite{rudin2012function_book}), we only need to show that $f$ is injective.  
        By Theorem \ref{thm: equivalent conditions for sufficiency} and Theorem \ref{thm: sufficient implies 1-1}, it is enough to show that 
        $$\Poisson_1(z,\xi) := \frac{|\Bergman_1(z,\xi)|^2}{\Bergman_1(z,z)}$$ 
        satisfies the condition $(3)$ in Theorem \ref{thm: equivalent conditions for sufficiency}.
    
        From the condition $f^* g_{B2} = \lambda g_{B1}$, we have
        $$ \partial \overline{\partial} \log \Bergman_1(z,z) 
        - \lambda \partial \overline{\partial} \log \Bergman_2(f(z),f(z)) = 0 $$
        for all $z \in \O_1$. 
        Then there exist a (simply-connected) open neighborhood $U \subset \O_1$ and a holomorphic function $\varphi$ on $U$ such that
        \begin{equation} \label{1.1}
            \log \Bergman_1(z,z) 
            - \lambda \log \Bergman_2(f(z),f(z)) = \varphi(z) + \overline{\varphi(z)}
        \end{equation}
        for all $z \in U$ and hence 
        $$ \log \Bergman_1(z,\xi) 
        - \lambda \log \Bergman_2(f(z),f(\xi)) = \varphi(z) + \overline{\varphi(\xi)}
        $$
        for all $z, \xi \in U$ (by shrinking $U$ if necessary).
        Furthermore, 
        \begin{equation} \label{1.2}
            \log \Poisson_1(z,\xi) 
            - \lambda \log \Poisson_2(f(z),f(\xi)) = \varphi(\xi) + \overline{\varphi(\xi)}
        \end{equation} 
        for all $z, \xi \in U$.
        Denote the left side of (\ref{1.1}) (after replacing the $z$ variable by $\xi$) and the left side of (\ref{1.2}) by $A(\xi)$ and $B(z,\xi)$, respectively. Then both $\exp (A(\xi))$ and $\exp (B(z,\xi))$ are well-defined on the whole domain $\O_1 \times \O_1$.
        Since they are real-analytic functions and coincide on $U \times U$, we have 
        %Note that, in (\ref{1.1}), since the left-hand side is well-defined on the whole domain $\O_1$, the right-hand side $\varphi(z) + \overline{\varphi(z)}$ can extend to a global function on $\O_1$ (this does not mean that $\varphi$ extends holomorphically). 
        \begin{align*}
            \Poisson_1(z, \xi) 
            &= \Poisson_2(f(z), f(\xi))^{\lambda} \frac{\Bergman_1(\xi,\xi)}{\Bergman_2(f(\xi), f(\xi))^{\lambda}} \\
            &= e^{-\lambda \Dia_2(f(z), f(\xi))} \cdot \Bergman_1(\xi,\xi)
        \end{align*}
        for all $z, \xi \in \O_1$, where $\Dia_2(w, \zeta) := \log \frac{\Bergman_2(w,w)\Bergman_2(\zeta, \zeta)}{|\Bergman_2(w,\zeta)|^2}$ is the Calabi's diastasis function(\cite{Calabi53}) for $\O_2$. 
        Hence, the condition $(3)$ in Theorem \ref{thm: equivalent conditions for sufficiency} is satisfied for $s(z, \zeta) := e^{-\lambda \Dia_2(f(z), \zeta)}$ and $t(\xi):= \Bergman_1(\xi,\xi)$
        and the proof is completed. 
        \hfill
        $\square$

\vspace{5mm}
        In contrast that $L^1$ conditions are imposed for the functions $s$ and $t$ in the condition (3) of Theorem~\ref{thm: equivalent conditions for sufficiency}, our $t(\xi) := \Bergman_1(\xi,\xi)$ is not a $L^1$ function on $\O_1$ in general.
        However, the $L^1$ conditions are not needed in our situation because we have the explicit formula (\ref{euq: kappa expression}) for the measure push-forward. We end the paper by giving a proof for $(3) \Rightarrow (2)$, which is mainly from Proposition 5.5 in \cite{ay2017information_book}.
        
        \begin{prop}
            Suppose that, for each $z \in \O_1$, there exist measurable functions $s(z, \cdot): \Xi_2 \rightarrow \RR$ and $t: \Xi_1 \rightarrow \RR$ such that
            \begin{align*}
                \Poisson_1(z, \xi) 
                = s(z, f(\xi)) t(\xi)
            \end{align*}
            almost everywhere for $\xi \in \O_1$.
            Then the function
            \begin{equation*}
                r(z, \xi) := \frac{P_1(z,\xi)}{Q(z, f(\xi))} 
            \end{equation*}
            does not depend on $z \in \O_1$, where $\kappa(P_1(z, \xi)dV(\xi))(z, \zeta) := Q(z, \zeta) \kappa(dV)(\zeta)$. 
        \end{prop}
        \begin{proof}
            From the assumption, 
            \begin{align*}
                \kappa(P_1(z, \xi) dV(\xi))(z, \zeta)
                &= \sum_{k=1}^{m} s(z,f(f^{-1}_k(\zeta))) t(f^{-1}_k(\zeta)) |J_{\CC}f^{-1}_k(\zeta)|^2 dV(\zeta) \\
                &= s(z,\zeta)  \sum_{k=1}^{m} t(f^{-1}_k(\zeta)) |J_{\CC}f^{-1}_k(\zeta)|^2 dV(\zeta)
            \end{align*}
            and
            \begin{align*}
                Q(z, \zeta) 
                := \frac{\kappa(P_1(z, \xi)dV(\xi))}{\kappa(dV)}
                = s(z,\zeta)  \frac{\sum_{k=1}^{m} t(f^{-1}_k(\zeta)) |J_{\CC}f^{-1}_k(\zeta)|^2}{\sum_{k=1}^{m} |J_{\CC}f^{-1}_k(\zeta)|^2   }.
            \end{align*}
            Therefore, $\frac{P_1(z,\xi)}{Q(z, f(\xi))}$ depends only on $\xi \in \O_1$.
        \end{proof}

	\bibliographystyle{abbrv}
	\bibliography{reference}

\begin{bibdiv}
\begin{biblist}

\bib{amari2000methods}{book}{
      author={Amari, Shunichi},
      author={Nagaoka, Hiroshi},
       title={Methods of information geometry},
   publisher={American Mathematical Soc.},
        date={2000},
      volume={191},
}

\bib{ay2017information_book}{book}{
      author={Ay, Nihat},
      author={Jost, J{\"u}rgen},
      author={V{\^a}n~L{\^e}, H{\^o}ng},
      author={Schwachh{\"o}fer, Lorenz},
       title={Information geometry},
   publisher={Springer},
        date={2017},
      volume={64},
}

\bib{Calabi53}{article}{
      author={Calabi, Eugenio},
       title={Isometric imbedding of complex manifolds},
        date={1953},
     journal={Ann. of Math. (2)},
      volume={58},
       pages={1\ndash 23},
}

\bib{cho2023statistical}{article}{
      author={Cho, Gunhee},
      author={Yum, Jihun},
       title={{Statistical Bergman geometry}},
        date={2023},
     journal={preprint arXiv:2305.10207},
}

\bib{greene2011geometry}{book}{
      author={Greene, Robert~E},
      author={Kim, Kang-Tae},
      author={Krantz, Steven~G},
       title={The geometry of complex domains},
   publisher={Springer Science \& Business Media},
        date={2011},
      volume={291},
}

\bib{Lu1966onkahler}{article}{
      author={Lu, Qi-Keng},
       title={{On K\"ahler manifolds with constant curvature (Chinese)}},
        date={1966},
     journal={Acta Math. Sinica},
      volume={16},
       pages={269\ndash 281},
}

\bib{rudin2012function_book}{book}{
      author={Rudin, Walter},
       title={Function theory in the unit ball of $\mathbb{C}^n$},
   publisher={Springer Science \& Business Media},
        date={2012},
      volume={241},
}

\bib{skwarczynski1980biholomorphic}{book}{
      author={Skwarczy{\'n}ski, Maciej},
       title={Biholomorphic invariants related to the {B}ergman functions},
   publisher={Instytut Matematyczny Polskiej Akademi Nauk (Warszawa)},
        date={1980},
}

\bib{yoo2017differential}{article}{
      author={Yoo, Sungmin},
       title={A differential-geometric analysis of the {B}ergman representative map},
        date={2017},
        ISSN={0066-2216,1730-6272},
     journal={Ann. Polon. Math.},
      volume={120},
      number={2},
       pages={163\ndash 181},
         url={https://doi.org/10.4064/ap170621-21-11},
}

\end{biblist}
\end{bibdiv}

\end{document}